\documentclass{amsart}

\usepackage[english]{babel}
\usepackage[T1]{fontenc}
\usepackage{latexsym}
\usepackage{ae}
\usepackage{amssymb}
\usepackage[hidelinks]{hyperref}
\usepackage[hmarginratio = 1:1]{geometry}

\makeatletter
\@namedef{subjclassname@2020}{%
  \textup{2020} Mathematics Subject Classification}
\makeatother

\newtheorem{theorem}{Theorem}[section]
\newtheorem{lemma}[theorem]{Lemma}
\newtheorem{corollary}[theorem]{Corollary}

\theoremstyle{definition}
\newtheorem{definition}[theorem]{Definition}

\theoremstyle{remark}
\newtheorem{remark}[theorem]{Remark}

\numberwithin{equation}{section}

\begin{document}

\title{Projective Analytic Vectors and Infinitesimal Generators}

\author{Rodrigo A. H. M. Cabral}
\thanks{This work is part of a PhD thesis \cite{rodrigotese}, which was supported by CNPq (Conselho Nacional de Desenvolvimento Cient\'ifico e Tecnol\'ogico).}
\address{Departamento de Matem\'atica, Instituto de Matem\'atica e Estat\'istica, Universidade de S\~ao Paulo (IME-USP), BR-05508-090, S\~ao Paulo, SP, Brazil.}
\email{rahmc@ime.usp.br; rodrigoahmc@gmail.com}

\dedicatory{}

\commby{}

\begin{abstract}
We establish the notion of a ``\textit{projective analytic vector}'', whose defining requirements are weaker than the usual ones of an analytic vector, and use it to prove generation theorems for one-parameter groups on locally convex spaces. More specifically, we give a characterization of the generators of strongly continuous one-parameter groups which arise as the result of a projective limit procedure, in which the existence of a dense set of projective analytic vectors plays a central role. An application to strongly continuous Lie group representations on Banach spaces is given, with a focused analysis on concrete algebras of functions and of pseudodifferential operators.
\end{abstract}



\maketitle

\section{Introduction}

In the Banach space context, there are very well-known generation theorems for one-parameter groups which rely on the hypothesis of existence of a dense set of analytic vectors. In Hilbert spaces, for example, E. Nelson proved in \cite[Lemma 5.1]{nelson} that a closed hermitian operator is self-adjoint if, and only if, it possesses a dense set of analytic vectors inside its domain -- generators of strongly continuous unitary one-parameter groups are all of the form $iT$, with $T$ a self-adjoint operator (see \cite[Theorem VIII.7, Theorem VIII.8]{reedsimon1}). J. Rusinek then improved this result in reference \cite{rusinek} to the Banach space context for an arbitrary strongly continuous one-parameter group -- see also \cite[Theorem 3.1.22]{bratteli1}.

In the paper \cite{magyarpaper}, Z. Magyar deals with generation theorems on locally convex spaces beyond the normed framework and provides, as a corollary, a classification of the generators of one-parameter groups with ``\textit{uniform exponential growth}'' -- see \cite[page 101]{magyarpaper}; see also \cite[Chapter 1]{magyarbook}. These groups are exactly the exponentially equicontinuous one-parameter groups which appear in Definition \ref{expequic}, below. In the present paper, we aim to characterize generators belonging to a different class of one-parameter groups: our groups are still of ``\textit{exponential growth}'', as defined on page 95 of reference \cite{magyarpaper}, but the uniformity requirement is relaxed. On the other hand, our groups, which were introduced by V.A. Babalola in \cite{babalola}, bear a strict relationship with the fixed fundamental system of seminorms generating the locally convex topology -- see Definition \ref{gamma}.

Our main goal is to prove new generation theorems in the locally convex realm with the aid of ``\textit{projective analytic vectors}'', a concept which fits very naturally into this framework. Then, a characterization of the generators of strongly continuous one-parameter groups belonging to a class defined by V.A. Babalola in reference \cite{babalola}, which includes all equicontinuous one-parameter groups, will be given in Theorem \ref{Theorem5}. Applications to equicontinuous one-parameter groups are given in Theorem \ref{Corollary6}, and an application to strongly continuous Lie group representations on Banach spaces is provided in Theorem \ref{examples}, with a subsequent investigation of certain algebras of functions and of pseudodifferential operators. In particular, we show that estimates which are similar to the ones provided in \cite[Proposition 4.2, page 262]{cordes} may be obtained, and with constants that do not depend on the dimension $n$ of the subjacent Euclidean space (the estimates in reference \cite{cordes} appear as a byproduct of Cordes' pseudodifferential classification of the smooth vectors under the action of the ($2n + 1$)-dimensional Heisenberg group by $*$-automorphisms). An analogous treatment is also given for the action of the $n$-dimensional torus. Our estimates highlight the strong interplay between the natural Fr\'echet topology of the algebra of smooth vectors and the partial derivatives of the symbols of these pseudodifferential operators.

\section{Preliminaries}

\textit{Some Conventions:} The notations $(\mathcal{X}, \Gamma)$ or $(\mathcal{X}, \tau)$ will frequently be used to indicate, respectively, that $\Gamma$ is a fundamental system of seminorms for the locally convex space $\mathcal{X}$ (in other words, a family of seminorms on $\mathcal{X}$ which generates its topology) or that $\tau$ is the locally convex topology of $\mathcal{X}$. The symbol $\mathcal{L}(\mathcal{X})$ will always denote the algebra of continuous linear operators defined on all of $\mathcal{X}$, and the term ``$\tau$-continuous'' will often be used to indicate that the continuity of an operator is related to the $\tau$ topology. A family of seminorms $\Gamma$ defined on a locally convex space $\mathcal{X}$ is said to be \textit{saturated} if, for any given finite subset $F$ of $\Gamma$, the seminorm defined by
\begin{equation*}p_F \colon x \longmapsto \max \left\{p(x): p \in F\right\}\end{equation*}
also belongs to $\Gamma$. A fundamental system of seminorms can always be enlarged to a saturated one by including the seminorms $p_F$, as defined above, in such a way that the resulting family generates the same topology. Hence, it will always be assumed in this paper that the families of seminorms to be considered are already saturated, whenever convenient, without further notice.

A linear operator $T$ on $\mathcal{X}$ will always be assumed as being defined on a vector subspace $\text{Dom }T$ of $\mathcal{X}$, called its \textit{domain}. When its domain is dense, the operator will be called \textit{densely defined}. Also, if $S$ and $T$ are two linear operators on $\mathcal{X}$, then their sum will be defined by 
\begin{equation*}\text{Dom }(S + T) := \text{Dom }S \cap \text{Dom }T, \quad (S + T)(x) := S(x) + T(x),\end{equation*} and their composition by
\begin{equation*}\text{Dom }(TS) := \left\{x \in \mathcal{X}: S(x) \in \text{Dom }T\right\}, \quad (TS)(x) := T(S(x)),\end{equation*} following the usual conventions of the classical theory of unbounded linear operators on Hilbert and Banach spaces. The \textit{range} of $T$ will be denoted by $\text{Ran }T$.

\begin{definition}[One-parameter Semigroups and Groups] \label{expequic}
A one-parameter semigroup on a Hausdorff locally convex space $(\mathcal{X}, \Gamma)$ is a family $\left\{V(t)\right\}_{t \geq 0}$ of linear operators in $\mathcal{L}(\mathcal{X})$ satisfying \begin{equation*}V(0) = I, \, V(s + t) = V(s)V(t), \qquad s, t \geq 0.\end{equation*} If, in addition, the semigroup satisfies the property that \begin{equation*}\lim_{t \rightarrow t_0} p(V(t)x - V(t_0)x) = 0, \qquad t_0 \geq 0, \, p \in \Gamma, \, x \in \mathcal{X},\end{equation*} then $\left\{V(t)\right\}_{t \geq 0}$ is called \textit{strongly continuous} or, more explicitly, a \textit{strongly continuous one-parameter semigroup}. Such a semigroup is called \textit{exponentially equicontinuous} \cite[Definition 2.1]{albanese} if there exists $a \geq 0$ satisfying the following property: for all $p \in \Gamma$, there exist $q \in \Gamma$ and $M_p > 0$ such that \begin{equation*}p(V(t)x) \leq M_p \, e^{at} q(x), \qquad t \geq 0, \, x \in \mathcal{X}.\end{equation*} If $a$ can be chosen equal to 0, such a semigroup will be called \textit{equicontinuous} \cite[page 234]{yosida}. A strongly continuous one-parameter semigroup is said to be \textit{locally equicontinuous} if, for every compact $K \subseteq [0, + \infty)$, the set $\left\{V(t): t \in K\right\}$ is equicontinuous, meaning that, for each $p \in \Gamma$, there exist $q \in \Gamma$ and $M_p > 0$ satisfying \begin{equation*} p(V(t)x) \leq M_p \, q(x), \qquad t \in K, \, x \in \mathcal{X}.\end{equation*} These definitions are analogous for one-parameter groups, switching from ``$\,t \geq 0$'' to ``$\,t \in \mathbb{R}$'', ``$\,e^{at}$'' to ``$\,e^{a|t|}$'' and ``$\,[0, + \infty)$'' to ``$\,\mathbb{R}$''.

In the group case, much of the above terminology can be adapted from one-parameter groups to general Lie groups. For example, if $G$ is a Lie group with unit $e$, then a family $\left\{V(g)\right\}_{g \in G}$ of linear operators in $\mathcal{L}(\mathcal{X})$ satisfying \begin{equation*}V(e) = I, \, V(gh) = V(g)V(h), \qquad g, h \in G\end{equation*} and \begin{equation*}\lim_{g \rightarrow h} V(g)x = V(h)x, \qquad x \in \mathcal{X}, \, h \in G,\end{equation*} is called a \textit{strongly continuous representation of $G$ on $\mathcal{X}$}. Such a group representation is called \textit{locally equicontinuous} if, for each compact $K \subseteq G$, the set $\left\{V(g): g \in K\right\}$ is equicontinuous.
\end{definition}

\begin{definition}[The Kernel Invariance Property (KIP)]
If $(\mathcal{X}, \Gamma)$ is a Hausdorff locally convex space, define for each $p \in \Gamma$ the closed subspace \begin{equation*}N_p := \left\{x \in \mathcal{X}: p(x) = 0\right\},\end{equation*} often referred to as the \textit{kernel} of the seminorm $p$, and the quotient map $\pi_p \colon \mathcal{X} \ni x \longmapsto [x]_p \in \mathcal{X}/N_p$. Then, $\mathcal{X}/N_p$ is a normed space with respect to the norm $\|[x]_p\|_p := p(x)$, and is not necessarily complete. Denote its completion by $\mathcal{X}_p := \overline{\mathcal{X}/N_p}$. A densely defined linear operator $T \colon \text{Dom }T \subseteq \mathcal{X} \longrightarrow \mathcal{X}$ is said to possess the \textit{kernel invariance property (KIP)} with respect to $\Gamma$ if it leaves their seminorms' kernels invariant, that is, \begin{equation*}T \, [\text{Dom }T \cap N_p] \subseteq N_p, \qquad p \in \Gamma\end{equation*} -- reference \cite{babalola} calls them ``compartmentalized operators''. If this property is fulfilled, then the linear operators \begin{equation*}T_p \colon \pi_p[\text{Dom }T] \subseteq \mathcal{X}_p \longrightarrow \mathcal{X}_p, \qquad T_p \colon [x]_p \longmapsto [T(x)]_p,\end{equation*} on the quotients are well-defined, and their domains are dense in $\mathcal{X}_p$.
\end{definition}

Next, we define the strongly continuous one-parameter (semi)groups which will be studied in this article.

\begin{definition}[$\Gamma$-semigroups and $\Gamma$-groups] \label{gamma}
Let $(\mathcal{X}, \Gamma)$ be a complete Hausdorff locally convex space. For each $p \in \Gamma$, define \begin{equation*}V_p := \left\{x \in \mathcal{X}: p(x) \leq 1\right\}.\end{equation*} Following \cite{babalola}, the following conventions will be used:

\begin{enumerate}

\item $\mathcal{L}_\Gamma(\mathcal{X}) \subseteq \mathcal{L}(\mathcal{X})$ denotes the family of continuous linear operators $A$ on $\mathcal{X}$ satisfying the property that, for all $p \in \Gamma$, there exists $\lambda(p, A) > 0$ such that \begin{equation*}A[V_p] \subseteq \lambda(p, A) \, V_p\end{equation*} or, equivalently, \begin{equation*}p(Ax) \leq \lambda(p, A) \, p(x), \qquad x \in \mathcal{X}.\end{equation*}

\item A strongly continuous one-parameter semigroup $t \longmapsto V(t)$ is said to be a $\Gamma$\textit{-semigroup} if, for each $p \in \Gamma$, there exist $M_p, \sigma_p \in \mathbb{R}$ such that $p(V(t)x) \leq M_p \, e^{\sigma_p t} p(x)$, for all $x \in \mathcal{X}$ and $t \geq 0$ -- in \cite{babalola} they are called (C$_0$, 1) semigroups; similarly for groups. $\Gamma$\textit{-groups} are defined in an analogous way, but with $p(V(t)x) \leq M_p \, e^{\sigma_p |t|} p(x)$, for all $x \in \mathcal{X}$ and $t \in \mathbb{R}$.

\end{enumerate}
\end{definition}

Note that the definition above automatically implies local equicontinuity of the one-parameter (semi)group $V$, and that the operators $V(t)$ have the (KIP) with respect to $\Gamma$. In particular, if $\left\{T_\alpha\right\}_{\alpha \in \mathcal{A}}$ is a net of linear operators having the (KIP) with respect to $\Gamma$, all defined on the same domain $\mathcal{D} \subseteq \mathcal{X}$, such that $T_\alpha(x)$ converges to $T(x)$, for all $x \in \mathcal{D} \subseteq \text{Dom }T$, then $T|_\mathcal{D}$ also has the (KIP) with respect to $\Gamma$. Since, by definition, the generator $T$ of a $\Gamma$-(semi)group $t \longmapsto V(t)$ satisfies \begin{equation*} \lim_{t \rightarrow 0} \frac{V(t)x - x}{t} = Tx,\end{equation*} for all $x \in \text{Dom }T$ (see Definition \ref{smooth}, below), it follows that $T$ has the (KIP) with respect to $\Gamma$.

The definitions of $\Gamma$-semigroups and $\Gamma$-groups given here are equivalent to the ones given in \cite[Definitions 2.1 and 2.2]{babalola}, as is proved in \cite[Theorem 2.6]{babalola} (note that a ``local equicontinuity-type'' requirement which is present in \cite[Definition 2.1]{babalola} is missing in \cite[Definition 2.2]{babalola}). Exploring the (KIP), the author of \cite{babalola} associates to the $\Gamma$-semigroup $t \longmapsto V(t)$ a family of strongly continuous semigroups $\left\{\tilde{V}_p\right\}_{p \in \Gamma}$, where $\tilde{V}_p$ is defined on the Banach space $(\mathcal{X}_p, \|\, \cdot \,\|_p)$. For each $p \in \Gamma$, the \textit{type} of $\tilde{V}_p$ \cite[page 306]{hille}, denoted by $w_p$, is defined as \begin{equation*}w_p := \inf_{t > 0} \, \frac{1}{t} \, \text{log }\|\tilde{V}_p(t)\|_p,\end{equation*} and may be a real number or $- \infty$ \cite[pages 40, 251]{engel}, \cite[Theorem 7.6.1]{hille}. The family $\left\{w_p\right\}_{p \in \Gamma}$ of elements in the extended real line is called \textit{the type of $V$}. Following \cite[page 170]{babalola}, if the supremum $w := \sup_{p \in \Gamma} w_p$ is a real number, then $V$ is said to be of \textit{bounded type} $w$ -- analogously, substituting ``$t > 0$'' by ``$t \in \mathbb{R} \backslash \left\{0\right\}$'' and ``$1/t$'' by ``$1/|t|$'', one defines $\Gamma$-groups of bounded type.

A nice aspect of $\Gamma$-semigroups of bounded type $w$ is that they satisfy the resolvent formula \begin{equation*}\text{R}(\lambda, T)x = \int_0^{+ \infty} e^{-\lambda t} V(t)x \, dt, \qquad x \in \mathcal{X},\end{equation*} for all $\lambda > w$, where $T$ is the infinitesimal generator of $V$, as is shown in \cite[Theorem 3.3]{babalola}.

Note that, unlike exponentially equicontinuous (semi)groups, the definitions of $\Gamma$-semigroups and $\Gamma$-groups \textit{depend} on the choice of the fundamental system of seminorms $\Gamma$.

Another very important concept is that of a resolvent set, so let $T \colon \text{Dom }T \subseteq \mathcal{X} \longrightarrow \mathcal{X}$ be a linear operator. Then, the set \begin{equation*}\rho(T) := \left\{\lambda \in \mathbb{C}: (\lambda I - T) \text{ is bijective and } (\lambda I - T)^{-1} \in \mathcal{L}(\mathcal{X})\right\}\end{equation*} is called the \textit{resolvent set} of the operator $T$ (note that it may happen that $\rho(T) = \emptyset$, or even $\rho(T) = \mathbb{C}$), and the map \begin{equation*}\text{R}(\, \cdot \, , T) \colon \rho(T) \ni \lambda \longmapsto \text{R}(\lambda, T) := (\lambda I - T)^{-1} \in \mathcal{L}(\mathcal{X})\end{equation*} is called the \textit{resolvent of $T$}. When $\lambda \in \rho(T)$, it will be said that \textit{the resolvent operator of $T$ exists at $\lambda$}, and $\text{R}(\lambda, T)$ is called the \textit{resolvent operator of $T$ at $\lambda$} or the \textit{$\lambda$-resolvent of $T$}. Finally, the complementary set \begin{equation*}\sigma(T) := \mathbb{C} \, \backslash \rho(T)\end{equation*} is called the \textit{spectrum} of the operator $T$.

\begin{definition}[Infinitesimal Generators, Smooth and Analytic Vectors] \label{smooth}
Let $t \longmapsto V(t)$ be a strongly continuous one-parameter semigroup on the Hausdorff locally convex space $(\mathcal{X}, \Gamma)$ and consider the subspace of vectors $x \in \mathcal{X}$ such that the limit \begin{equation*}\lim_{t \rightarrow 0} \frac{V(t)x - x}{t}\end{equation*} exists. Then, the linear operator $T$ defined by \begin{equation*}\text{Dom }T := \left\{x \in \mathcal{X}: \lim_{t \rightarrow 0} \frac{V(t)x - x}{t} \text{ exists in } \mathcal{X}\right\}\end{equation*} and $T(x) := \lim_{t \rightarrow 0} \frac{V(t)x - x}{t}$ is called the \textit{infinitesimal generator} (or, simply, the generator) of the semigroup $t \longmapsto V(t)$ (the definition for groups is analogous). Also, if $G$ is a Lie group and $g \longmapsto V(g)$ is a strongly continuous representation of $G$ on $\mathcal{X}$ then, for each fixed element $X$ of its Lie algebra $\mathfrak{g}$, the infinitesimal generator of the one-parameter group $t \longmapsto V(\exp tX)$ will be denoted by $dV(X)$ ($\exp$ denotes the exponential map of the Lie group $G$). A vector $x \in \mathcal{X}$ is called a \textit{C$^\infty$ vector for $V$, or a smooth vector for $V$}, if the map $G \ni g \longmapsto V(g)x$ is of class $C^\infty$: a map $f \colon G \longrightarrow \mathcal{X}$ is of class $C^\infty$ at $g \in G$ if it possesses continuous partial derivatives of all orders with respect to a chart around $g$. The subspace of smooth vectors for $V$ will be denoted by $C^\infty(V)$. Moreover, following \cite[page 54]{moore}, we will say that a vector $x \in \mathcal{X}$ is \textit{analytic for $V$} if $x \in C^\infty(V)$ and the map $F_x \colon G \ni g \longmapsto V(g)x$ is analytic: in other words, if $x \in C^\infty(V)$ and, for each $g \in G$ and every analytic chart $h \colon g' \longmapsto (t_k(g'))_{1 \leq k \leq d}$ around $g$ sending it to 0 there exists $r_x > 0$ such that \begin{equation*}\sum_{\alpha \in \mathbb{N}^d} \frac{p(\partial^\alpha (F_x \circ h^{-1})(0))}{\alpha!} \, |t(g')^\alpha| < + \infty\end{equation*} and \begin{equation*} \sum_{\alpha \in \mathbb{N}^d} \frac{\partial^\alpha (F_x \circ h^{-1})(0)}{\alpha!} \, t(g')^\alpha = F_x(g'), \end{equation*} for every $p \in \Gamma$, whenever $|t(g')| < r_x$, where \begin{equation*} t(g')^\alpha = t_1(g')^{\alpha_1} \ldots t_d(g')^{\alpha_d} \quad \text{and} \quad |t(g')| := \max_{1 \leq k \leq d} |t_k(g')| \end{equation*} -- note that $r_x$ is independent of $p \in \Gamma$. The subspace of analytic vectors for $V$ will be denoted by $C^\omega(V)$. If $\tau$ is the topology defined by $\Gamma$, then the elements of $C^\omega(V)$ will sometimes be called \textit{$\tau$-analytic}.

Now, let $T$ be a linear operator on $\mathcal{X}$. An element $x \in \mathcal{X}$ is called a \textit{$\tau$-analytic vector for $T$} if \begin{equation*}x \in C^\infty(T) := \bigcap_{n = 1}^{+ \infty} \text{Dom }T^n\end{equation*} and there exists $r_x > 0$ such that \begin{equation*}\sum_{n \geq 0} \frac{p\,(T^n(x))}{n!} |u|^n < \infty, \qquad |u| < r_x,\end{equation*} for every $p \in \Gamma$. The subspace of analytic vectors for $T$ will be denoted by $C^\omega(T)$. See also the definition of ``\textit{$s$-analytic vectors}'' of an operator on page 94 of reference \cite{magyarpaper}.
\end{definition}

A linear operator $T \colon \text{Dom T} \subseteq \mathcal{X} \longrightarrow \mathcal{X}$ on the Hausdorff locally convex space $\mathcal{X}$ is \textit{closed} if its graph is a closed subspace of $\mathcal{X} \times \mathcal{X}$; an operator $S \colon \text{Dom S} \subseteq \mathcal{X} \longrightarrow \mathcal{X}$ is \textit{closable} if it has a closed extension or, equivalently, if for every net $\left\{x_\alpha\right\}_{\alpha \in \mathcal{A}}$ in $\text{Dom }S$ such that $x_\alpha \longrightarrow 0$ and $S(x_\alpha) \longrightarrow y$, one has $y = 0$. If $S$ is closable, then it has a minimal closed extension, called the \textit{closure} of $S$ and denoted by $\overline{S}$. Two very important results regarding infinitesimal generators on locally convex spaces are Propositions 1.3 and 1.4 of \cite{komura}, which prove that infinitesimal generators of strongly continuous locally equicontinuous one-parameter semigroups on sequentially complete locally convex spaces are densely defined and closed.

\begin{definition}[Cores]
If $T \colon \text{Dom }T \subseteq \mathcal{X} \longrightarrow \mathcal{X}$ is a closed linear operator on the Hausdorff locally convex space $\mathcal{X}$ and $\mathcal{D} \subseteq \text{Dom }T$ is a linear subspace of $\text{Dom }T$ such that \begin{equation*}\overline{T|_\mathcal{D}} = T,\end{equation*} then $\mathcal{D}$ is called a \textit{core} for $T$.
\end{definition}

Using the argument at the end of the proof of \cite[Lemma 4]{magyarpaper} we arrive at the following lemma:

\begin{lemma}\label{Lemma}
Let $t \longmapsto V(t)$ be a strongly continuous locally equicontinuous one-parameter group on a complete Hausdorff locally convex space $(\mathcal{X}, \Gamma)$ having $T$ as its generator. Suppose $\mathcal{D}$ is a dense subspace of $\mathcal{X}$ which is contained in $C^\infty(V)$ and is group invariant, that is, \begin{equation*}V(t)[\mathcal{D}] \subseteq \mathcal{D}, \qquad t \in \mathbb{R}.\end{equation*} If $n \in \mathbb{N}$ and $T^n$, $T^n|_\mathcal{D}$ are both closable operators, then $\overline{T^n} = \overline{T^n|_\mathcal{D}}$.
\end{lemma}

A corollary of Lemma \ref{Lemma} is that, under these hypotheses, $\mathcal{D}$ is a core for the generator $T$. An analogous lemma is also available for one-parameter semigroups.

Lemma \ref{Lemma} will be used in the proof of Theorem \ref{Theorem1}.

\begin{definition}[Projective Limits]
Let $\left\{\mathcal{X}_i\right\}_{i \in I}$ be a family of Hausdorff locally convex spaces, where $I$ is a directed set under the partial order $\preceq$ and, for all $i, j \in I$ satisfying $i \preceq j$, suppose that there exists a continuous linear map $\mu_{ij} \colon \mathcal{X}_j \longrightarrow \mathcal{X}_i$ satisfying $\mu_{ij} \circ \mu_{jk} = \mu_{ik}$, whenever $i \preceq j \preceq k$ ($\mu_{ii}$ is, by definition, the identity map, for all $i \in I$). Then, $(\mathcal{X}_i, \mu_{ij}, I)$ is called a \textit{projective system (or an inverse system)} of Hausdorff locally convex spaces. Consider the vector space $\mathcal{X}$ defined by \begin{equation*}\mathcal{X} := \left\{(x_i)_{i \in I} \in \prod_{i \in I} \mathcal{X}_i: \mu_{ij}(x_j) = x_i, \text{ for all }i, j \in I \text{ with }i \preceq j\right\}\end{equation*} and equipped with the relative Tychonoff's product topology or, equivalently, the coarsest topology for which every canonical projection $\pi_j \colon (x_i)_{i \in I} \longmapsto x_j$ is continuous, relativized to $\mathcal{X}$. Then, $\mathcal{X}$ is a Hausdorff locally convex space, and is called the \textit{projective limit (or inverse limit)} of the family $\left\{\mathcal{X}_i\right\}_{i \in I}$. In this case, the notation $\varprojlim \mathcal{X}_i := \mathcal{X}$ is employed.
\end{definition}

Any complete Hausdorff locally convex space $(\mathcal{X}, \Gamma)$ is isomorphic to the projective limit of Banach spaces $\varprojlim \mathcal{X}_p$, so they will be treated as if they were the same space, in the proofs.

Now, suppose there are linear operators $T_i \colon \text{Dom }T_i \subseteq \mathcal{X}_i \longrightarrow \mathcal{X}_i$ which are connected by the relations $\mu_{ij} [\text{Dom }T_j] \subseteq \text{Dom }T_i$ and $T_i \circ \mu_{ij} = \mu_{ij} \circ T_j$, whenever $i \preceq j$. Then, the family $\left\{T_i\right\}_{i \in I}$ is said to be a \textit{projective family of linear operators}. The latter relation ensures that the linear transformation $T$ defined on $\text{Dom }T := \varprojlim \text{Dom }T_i \subseteq \varprojlim \mathcal{X}_i$ by $T(x_i)_{i \in I} := (T_i(x_i))_{i \in I}$ has its range inside $\varprojlim \mathcal{X}_i$, thus defining a linear operator on $\varprojlim \mathcal{X}_i$. This operator is called the \textit{projective limit of $\left\{T_i\right\}_{i \in I}$}, as in \cite[page 167]{babalola}.

\section{Main results}

The next task will be to define projective analytic vectors on locally convex spaces, so that some useful theorems become available:

\begin{definition}[Projective Analytic Vectors]
Let $(\mathcal{X}, \tau)$ be a Hausdorff locally convex space with a fundamental system of seminorms $\Gamma$ and $T$ be a linear operator defined on $\mathcal{X}$. An element $x \in \mathcal{X}$ is called a \textit{$\tau$-projective analytic vector for $T$} if $x \in C^\infty(T)$ and, for every $p \in \Gamma$, there exists $r_{x, p} > 0$ such that \begin{equation*}\sum_{n \geq 0} \frac{p\,(T^n(x))}{n!} |u|^n < \infty, \qquad |u| < r_{x, p}.\end{equation*}
\end{definition}

Note that, for this definition to make sense, it is necessary to show that it does not depend on the choice of the particular system of seminorms: if $x$ is $\tau$-projective analytic with respect to $\Gamma$ and $\Gamma'$ is another saturated family of seminorms generating the topology of $\mathcal{X}$ then, for each $q' \in \Gamma'$, there exists $C_{q'} > 0$ and $q \in \Gamma$ such that $q'(y) \leq C_{q'} \, q(y)$, for all $y \in \mathcal{X}$. Therefore, making $r_{x, q'} := r_{x, q}$, one obtains for every $u \in \mathbb{C}$ satisfying $|u| < r_{x, q'}$ that \begin{equation*}\sum_{n \geq 0} \frac{q'(T^n(x))}{n!} |u|^n \leq C_{q'} \sum_{n \geq 0} \frac{q \,(T^n(x))}{n!} |u|^n < \infty.\end{equation*} By symmetry, the assertion is proved. This motivates the use of the notation ``$\tau$-projective analytic'' to indicate that the analytic vector in question is related to the topology $\tau$ (when there is no danger of confusion, the symbol $\tau$ will be omitted). The subspace formed by all of the projective analytic vectors for $T$ is going to be denoted by $C^\omega_\leftarrow(T)$. The prefix ``projective'' stands for the fact that $C^\omega_\leftarrow(T)$ can be seen as a dense subspace of $\varprojlim \pi_p[C^\omega_\leftarrow(T)]$ via the canonical map $x \longmapsto ([x]_p)_{p \in \Gamma}$ and, if $T$ has the (KIP) with respect to $\Gamma$, then $\pi_p[C^\omega_\leftarrow(T)]$ consists entirely of analytic vectors for $T_p$, for every $p \in \Gamma$ -- note that the projective limit is well-defined, since the family $\left\{\pi_p[C^\omega_\leftarrow(T)]\right\}_{p \in \Gamma}$ gives rise to a canonical projective system.

Differently of what is required for analytic vectors, \textit{no uniformity in $p$} is asked in the above definition. Hence, the requirements just imposed are weaker that the usual ones, so the subspace of $\tau$-analytic vectors $C^\omega(T)$ satisfies the inclusion $C^\omega(T) \subseteq C^\omega_\leftarrow(T)$.

\begin{theorem}\label{Theorem1}
Let $(\mathcal{X}, \Gamma)$ be a complete Hausdorff locally convex space and $T$ be a linear operator on $\mathcal{X}$ having the (KIP) with respect to $\Gamma$. Suppose that $T$ has a dense set of projective analytic vectors and that, for each $p \in \Gamma$, there exist two real numbers $\sigma_p \geq 0$ and $M_p \geq 1$ satisfying \begin{equation} \label{general} p((\lambda I - T)^n x) \geq M_p^{-1} (|\lambda| - \sigma_p)^{n} p(x), \qquad x \in \text{Dom }T^n,\end{equation} for all $|\lambda| > \sigma_p$ and $n \in \mathbb{N}$. Then, $T$ is closable and $\overline{T}$ is the generator of a $\Gamma$-group $t \longmapsto V(t)$ satisfying $p(V(t)x) \leq M_p \, e^{\sigma_p |t|} p(x)$, for all $p \in \Gamma$, $x \in \mathcal{X}$ and $t \in \mathbb{R}$.
\end{theorem}

\begin{proof}
As a consequence of the (KIP) combined with the density hypothesis, the linear operator $T_p$ induced on the quotient $\mathcal{X}/N_p$ possesses a dense subspace $\pi_p[C^\omega_\leftarrow(T)]$ of analytic vectors, for each fixed $p \in \Gamma$. Hence, combining the estimates (\ref{general}) with \cite[Theorem 1, Theorem 2]{rusinek} guarantees that $T_p$ is closable and $\overline{T_p}$ is the generator of a strongly continuous one-parameter group $t \longmapsto \tilde{V}_p(t)$ on $\mathcal{X}_p$, for all $p \in \Gamma$. Moreover, there exist numbers $\sigma_p \geq 0$ and $M_p \geq 1$ satisfying the inequality $\|\tilde{V}_p(t)[x]_p\|_p \leq M_p \, e^{\sigma_p |t|} \|[x]_p\|_p$, for all $p \in \Gamma$, $x \in \mathcal{X}$ and $t \in \mathbb{R}$. The arguments in the proof of \cite[Theorem 4.2]{babalola} show that $\left\{\tilde{V}_p(t)\right\}_{p \in \Gamma}$ is a projective family of operators, for each fixed $t \in \mathbb{R}$. Defining $V(t)$ as the projective limit of $\left\{\tilde{V}_p(t)\right\}_{p \in \Gamma}$ yields a $\Gamma$-group $t \longmapsto V(t)$ on $\mathcal{X}$, and the projective limit $\tilde{T}$ of the projective family $\left\{\overline{T_p}\right\}_{p \in \Gamma}$ is the generator of $t \longmapsto V(t)$, by \cite[Theorem 2.5]{babalola} (adapted to one-parameter groups). Note that $T \subset \tilde{T}$. In particular, this shows that $T$ is closable and \begin{equation*}\overline{T} \subset \tilde{T},\end{equation*} since $\tilde{T}$ is closed. Therefore, since $V$ leaves the dense subspace $C^\omega_\leftarrow(T) \subseteq C^\infty(\tilde{T})$ invariant, it follows from Lemma \ref{Lemma} above that it is a core for $\tilde{T}$, so \begin{equation*}\tilde{T} = \overline{\tilde{T}|_{C^\omega_\leftarrow(T)}} = \overline{T|_{C^\omega_\leftarrow(T)}} \subset \overline{T}.\end{equation*} This establishes the result.
\end{proof}

\begin{remark}
Note that it was also proved that, under the hypotheses of Theorem \ref{Theorem1}, $C^\omega_\leftarrow(T)$ is a core for $\overline{T}$.
\end{remark}

\begin{remark}\label{corollarysemi}
A similar theorem for $\Gamma$-semigroups is also true: if one substitutes the estimates (\ref{general}) by $p((\lambda I - T)^n x) \geq M_p^{-1} (\lambda - \sigma_p)^{n} p(x)$, for all $x \in \text{Dom }T^n$, $\lambda > \sigma_p$ and $n \in \mathbb{N}$, then it follows that $T$ is closable and $\overline{T}$ is the generator of a $\Gamma$-semigroup $t \longmapsto V(t)$ satisfying $p(V(t)x) \leq M_p \, e^{\sigma_p t} p(x)$, for all $p \in \Gamma$, $x \in \mathcal{X}$ and $t \geq 0$. To reach this conclusion, two adaptations in the proof of Theorem \ref{Theorem1} must be done: (i) by invoking \cite[Theorem 1]{magyarpaper}, instead of \cite[Theorem 1, Theorem 2]{rusinek}, one concludes that each $\overline{T_p}$ is the generator of a strongly continuous one-parameter semigroup on $\mathcal{X}_p$; (ii) then, using a semigroup version of Lemma \ref{Lemma}, one finally reaches the conclusion that $\overline{T}$ is the generator of a $\Gamma$-semigroup on $\mathcal{X}$.
\end{remark}

\begin{corollary}\label{Corollary2}
Let $(\mathcal{X}, \Gamma)$ be a complete Hausdorff locally convex space and $T$ be a linear operator on $\mathcal{X}$ having the (KIP) with respect to $\Gamma$. Suppose that $T$ has a dense set of analytic vectors and that, for each $p \in \Gamma$, there exist two real numbers $\sigma_p \geq 0$ and $M_p \geq 1$ satisfying \begin{equation*}p((\lambda I - T)^n x) \geq M_p^{-1} (|\lambda| - \sigma_p)^{n} p(x), \qquad x \in \text{Dom }T^n,\end{equation*} for all $|\lambda| > \sigma_p$ and $n \in \mathbb{N}$. Then, $T$ is closable and $\overline{T}$ is the generator of a $\Gamma$-group $t \longmapsto V(t)$ satisfying $p(V(t)x) \leq M_p \, e^{\sigma_p |t|} p(x)$, for all $p \in \Gamma$, $x \in \mathcal{X}$ and $t \in \mathbb{R}$.
\end{corollary}

\begin{proof}
Follows at once from Theorem \ref{Theorem1}, since $C^\omega(T) \subseteq C^\omega_\leftarrow(T)$.
\end{proof}

In view of Remark \ref{corollarysemi}, Corollary \ref{Corollary2} also has an analogous semigroup version. A converse statement of Theorem \ref{Theorem1} is also true:

\begin{theorem}\label{Theorem3}
Let $(\mathcal{X}, \Gamma)$ be a complete Hausdorff locally convex space and $T$ be a linear operator on $\mathcal{X}$ which is the generator of a $\Gamma$-group. Then, $T$ has a dense set of projective analytic vectors and, for each $p \in \Gamma$, there exist two real numbers $\sigma_p \geq 0$ and $M_p \geq 1$ such that \begin{equation*}p((\lambda I - T)^n x) \geq M_p^{-1} (|\lambda| - \sigma_p)^{n} p(x), \qquad x \in \text{Dom }T^n,\end{equation*} for all $|\lambda| > \sigma_p$ and $n \in \mathbb{N}$.
\end{theorem}

\begin{proof}
By \cite[Theorem 4.2]{babalola}, there exists a projective family $\left\{T_p\right\}_{p \in \Gamma}$ of closable linear operators such that $\overline{T_p}$ is the generator of a strongly continuous one-parameter group $\tilde{V}_p \colon t \longmapsto \tilde{V}_p(t)$ on $\mathcal{X}_p$, for each $p \in \Gamma$, with $T = \varprojlim \overline{T_p}$ being the generator of the $\Gamma$-group $V \colon t \longmapsto V(t)$ defined by $V(t) := \varprojlim \tilde{V}_p(t)$, $t \in \mathbb{R}$. Also, for each $p \in \Gamma$, there exist $\sigma_p \geq 0$ and $M_p \geq 1$ satisfying $p(V(t)x) \leq M_p \, e^{\sigma_p |t|} p(x)$, for all $x \in \mathcal{X}$ and $t \in \mathbb{R}$. Applying \cite[Theorem 4]{nelson} and \cite[Theorem 1.4]{poulsen} to the one-dimensional Lie group $\mathbb{R}$ (see also \cite{bratteliheat}), it follows that $\overline{T_p^2}$ is the generator of a strongly continuous one-parameter semigroup $t \longmapsto S_p(t)$ on $\mathcal{X}_p$ satisfying \begin{equation*}S_p(t)[\mathcal{X}_p] \subseteq C^\omega(\overline{T_p}), \qquad t > 0,\end{equation*} for each $p \in \Gamma$. Hence, \begin{equation} \label{S_p(t)} \bigcup_{t > 0} S_p(t)[\mathcal{X}_p] \subseteq C^\omega(\overline{T_p}), \qquad p \in \Gamma.\end{equation} Repeating the argument made in the proof of \cite[Theorem 4.2]{babalola} on formula \begin{equation*}S_p(t)x_p = \lim_{\lambda \rightarrow +\infty} e^{-\lambda t} \sum_{k = 0}^{+ \infty} \frac{(\lambda t)^k [\lambda \text{R}(\lambda, \overline{T_p^2})]^k}{k!} \, (x_p), \quad p \in \Gamma, \, t \geq 0, \, x_p \in \mathcal{X}_p,\end{equation*} (this formula may also be found in \cite[(11.7.2), page 352]{hille}) one sees that $\left\{S_p(t)\right\}_{p \in \Gamma}$ is a projective family of linear operators, for each fixed $t \geq 0$. Hence, the projective limit semigroup \begin{equation*}S \colon t \longmapsto S(t) := \varprojlim S_p(t), \qquad t \geq 0,\end{equation*} on $\varprojlim \mathcal{X}_p$ is well-defined.

Next, we show that $S(t)[\mathcal{X}] \subseteq C^\infty(T)$, for all $t > 0$. First, note that for all $x \in \mathcal{X}$, the function $t \longmapsto V(t)x$ is infinitely differentiable on $[0, +\infty)$ if, and only if, $t \longmapsto \tilde{V}_p(t)([x]_p)$ is infinitely differentiable on $[0, +\infty)$, for all $p \in \Gamma$. In other words, $x \in \mathcal{X}$ belongs to $C^\infty(T)$ if, and only if, $[x]_p$ belongs to $C^\infty(\overline{T_p})$, for every $p \in \Gamma$. Fix $t > 0$ and $x \in \mathcal{X}$. In view of $S_p(t)[\mathcal{X}_p] \subseteq C^\infty(\overline{T_p})$, for all $p \in \Gamma$, one concludes that $[S(t)x]_p = S_p(t)([x]_p)$ belongs to $C^\infty(\overline{T_p})$, for all $p \in \Gamma$. Therefore, $S(t)x$ must belong to $C^\infty(T)$, proving the inclusion $S(t)[\mathcal{X}] \subseteq C^\infty(T)$, for all $t > 0$. Combining this conclusion with (\ref{S_p(t)}) yields \begin{equation*}\bigcup_{t > 0} S(t)[\mathcal{X}] \subseteq C^\omega_\leftarrow(T).\end{equation*} Since $S \colon t \longmapsto S(t)$ is strongly continuous \cite[Lemma 7 b), page 26]{moore} the union above is dense in $\mathcal{X}$. This shows the density of $C^\omega_\leftarrow(T)$ in $\mathcal{X}$.

The claimed estimates follow at once from \cite[Theorem 4.2]{babalola}.
\end{proof}

Still under the hypotheses of Theorem \ref{Theorem3}, it is possible to show that $\overline{T^2}$ is the generator of a $\Gamma$-semigroup:

\begin{corollary}\label{Corollary4}
Let $(\mathcal{X}, \Gamma)$ be a complete Hausdorff locally convex space and $T$ be a linear operator on $\mathcal{X}$ which is the generator of a $\Gamma$-group. Then, $\overline{T^2}$ is the generator of a $\Gamma$-semigroup $S \colon t \longmapsto S(t)$ and $\mathcal{X}_0 := \left\{S(t)x: x \in \mathcal{X}, t > 0\right\}$ is a dense subspace of $\mathcal{X}$ which is a core for $\overline{T^2}$. Since $\mathcal{X}_0 \subseteq C^\omega_\leftarrow(T) \subseteq C^\infty(T) \subseteq C^\infty(\overline{T^2})$, the dense subspaces $C^\infty(T)$ and $C^\omega_\leftarrow(T)$ are also cores for $\overline{T^2}$. \end{corollary}

\begin{proof}
Let $S \colon t \longmapsto S(t)$ be the projective limit semigroup defined in Theorem \ref{Theorem3}. By the proof of Theorem \ref{Theorem3}, $C^\infty(T)$ is a dense subspace of $\mathcal{X}$ which is invariant by the operators $S(t)$, for all $t > 0$. Then, since $\overline{T^2} \subset \varprojlim \overline{T_p^2} =: \tilde{T}_2$, an application of Lemma \ref{Lemma} yields \begin{equation*} \tilde{T}_2 = \overline{\tilde{T}_2|_{C^\infty(T)}} = \overline{T^2|_{C^\infty(T)}} \subset \overline{T^2}, \end{equation*} so $\overline{T^2} = \varprojlim \overline{T_p^2}$. Also as a consequence of the proof of Theorem \ref{Theorem3}, the subspace \begin{equation*} \mathcal{X}_0 := \left\{S(t)x: x \in \mathcal{X}, t > 0\right\} \end{equation*} is dense in $\mathcal{X}$, so another application of Lemma \ref{Lemma} shows that $\mathcal{X}_0$ is a core for $\overline{T^2}$.
\end{proof}

Therefore, the following theorem holds:

\begin{theorem}\label{Theorem5}
Let $(\mathcal{X}, \Gamma)$ be a complete Hausdorff locally convex space, $T$ be a closed linear operator on $\mathcal{X}$ and, for each fixed $p \in \Gamma$, let $\sigma_p \geq 0$ and $M_p \geq 1$ be two real numbers. Then, $T$ is the generator of a $\Gamma$-group $t \longmapsto V(t)$ satisfying $p(V(t)x) \leq M_p \, e^{\sigma_p |t|} p(x)$, for all $p \in \Gamma$, $x \in \mathcal{X}$ and $t \in \mathbb{R}$ if, and only if, the following two conditions are satisfied:

\begin{enumerate}
\item $T$ has the (KIP) with respect to $\Gamma$ and, for each $p \in \Gamma$, \begin{equation*}p((\lambda I - T)^n x) \geq M_p^{-1} (|\lambda| - \sigma_p)^{n} p(x), \qquad x \in \text{Dom }T^n,\end{equation*} for all $|\lambda| > \sigma_p$, $n \in \mathbb{N}$;
\item $T$ has a dense set of projective analytic vectors.
\end{enumerate}
In this case, $C^\omega_\leftarrow(T)$ is a core for $T$ and $\overline{T^2}$ generates a $\Gamma$-semigroup $t \longmapsto S(t)$ such that the dense subspaces $\mathcal{X}_0 := \left\{S(t)x: x \in \mathcal{X}, t > 0\right\}$, $C^\omega_\leftarrow(T)$ and $C^\infty(T)$ are cores for $\overline{T^2}$.
\end{theorem}

\begin{remark}
We note here that, by \cite[Theorem 3]{magyarpaper}, if $T$ is the generator of a $\Gamma$-group on a complete Hausdorff locally convex space, then it has a dense set of analytic vectors. This fact, when combined with Corollary \ref{Corollary2} and \cite[Theorem 4.2]{babalola}, gives an analogous characterization, if ones substitutes the words ``\textit{projective analytic vectors}'' by ``\textit{analytic vectors}'', in the statement of Theorem \ref{Theorem5}. Actually, \cite[Theorem 3]{magyarpaper} gives a much stronger result, regarding \textit{entire vectors} (for this definition, see page 95 of that paper): \textit{if $T$ is the generator of a $\Gamma$-group on a (sequentially) complete Hausdorff locally convex space, then it has a dense set of entire vectors}. Therefore, since every entire vector is trivially a projective analytic vector, Theorem \ref{Theorem5} remains valid if one substitutes the words ``\textit{projective analytic vectors}'' by ``\textit{entire vectors}''.
\end{remark}

In order to apply these results for equicontinuous groups, some concepts will be introduced:

\begin{definition}[Dissipative and Conservative Operators] \label{dissipative}
Following \cite[Definition 3.9]{albanese}, a linear operator $T \colon \text{Dom }T \subseteq \mathcal{X} \longrightarrow \mathcal{X}$ on $(\mathcal{X}, \Gamma)$ will be called \textit{$\Gamma$-dissipative} if, for every $p \in \Gamma$, $\lambda > 0$ and $x \in \text{Dom }T$, \begin{equation*}p((\lambda I - T)x) \geq \lambda \, p(x).\end{equation*} If $\mathcal{X}$ is Hausdorff, this implies $\lambda I - T$ is an injective linear operator, for every $\lambda > 0$. $T$ is called \textit{$\Gamma$-conservative} if both $T$ and $-T$ are $\Gamma$-dissipative or, equivalently, if the inequality \begin{equation*}p((\lambda I - T)x) \geq |\lambda| \, p(x)\end{equation*} holds for all $p \in \Gamma$, $\lambda \in \mathbb{R} \backslash \left\{0\right\}$ and $x \in \text{Dom }T$.\end{definition}

Whenever convenient, the prefix $\Gamma$ will be omitted from the terminology. Note that the definitions of dissipativity and conservativity both depend on the particular choice of the fundamental system of seminorms -- see Remark 3.10 of \cite{albanese}, for an illustration of this fact. For any given generator $T$ of an equicontinuous one-parameter semigroup $t \longmapsto V(t)$ (on a Hausdorff sequentially complete locally convex space), Remark 3.12 of \cite{albanese} shows with a simple calculation that there always exists a fundamental system of seminorms $\Gamma$ for $\mathcal{X}$ with respect to which $T$ is $\Gamma$-dissipative (see also \cite[Remark 2.2(i)]{albanesemontel}). With respect to this $\Gamma$, one also has that \begin{equation} \label{contractive} p(V(t)x) \leq p(x), \qquad p \in \Gamma, \, t \geq 0, \, x \in \mathcal{X}.\end{equation} A $\Gamma$-semigroup satisfying (\ref{contractive}) will be called a \textit{$\Gamma$-contractive semigroup}. A small adaptation of this remark yields an analogous result regarding $\Gamma$-conservativity for generators of equicontinuous one-parameter groups: more precisely, for any equicontinuous one-parameter group $t \longmapsto V(t)$ there always exists a fundamental system of seminorms $\Gamma$ for $\mathcal{X}$ with respect to which its generator, $T$, is $\Gamma$-conservative and \begin{equation} \label{isometric} p(V(t)x) = p(x), \qquad p \in \Gamma, \, t \in \mathbb{R}, \, x \in \mathcal{X}.\end{equation} A $\Gamma$-group satisfying (\ref{isometric}) will be called a \textit{$\Gamma$-isometric group}.

Hence, the theorems above apply directly to equicontinuous one-parameter groups, with the appropriate choice of $\Gamma$, making $M_p = 1$ and $\sigma_p = 0$, for every $p \in \Gamma$. They give the following useful corollary:

\begin{theorem}\label{Corollary6}
Let $(\mathcal{X}, \tau)$ be a complete Hausdorff locally convex space and $T$ be a linear operator on $\mathcal{X}$. Then,

\begin{enumerate}
\item If $T$ is closed, $T$ will be the generator of an equicontinuous group if, and only if, it has a dense set of projective analytic vectors and there is a fundamental system of seminorms $\Gamma$ with respect to which $T$ has the (KIP) and $T$ is $\Gamma$-conservative.
\item Let $\Gamma$ be a fundamental system of seminorms for $\mathcal{X}$. If $T$ is, simultaneously, a $\Gamma$-dissipative operator and the generator of a $\Gamma$-semigroup, then it actually generates a $\Gamma$-contractive semigroup. Consequently, if $T$ is a $\Gamma$-conservative operator which is the generator of a $\Gamma$-group, then it actually generates a $\Gamma$-isometric group; in this case, $\overline{T^2}$ is the generator of a $\Gamma$-contractive semigroup.
\end{enumerate}
\end{theorem}

\begin{proof}
(1) ($\Rightarrow$) Suppose $T$ is the generator of an equicontinuous group. By the observations made right after Definition \ref{dissipative}, there exists a fundamental system of seminorms for $\mathcal{X}$ such that $T$ is $\Gamma$-conservative and $p(V(t)x) = p(x)$, for all $p \in \Gamma$, $t \in \mathbb{R}$ and $x \in \mathcal{X}$. Therefore, the result follows at once from Theorem \ref{Theorem5}.

($\Leftarrow$) Making $M_p = 1$ and $\sigma_p = 0$ in Theorem \ref{Theorem5} gives the conclusion that $T$ is the generator of a $\Gamma$-group $t \longmapsto V(t)$ on $\mathcal{X}$ satisfying $p(V(t)x) \leq p(x)$, for all $p \in \Gamma$, $t \in \mathbb{R}$ and $x \in \mathcal{X}$. Consequently, we also have that $p(x) = p(V(-t)V(t)x) \leq p(V(t)x)$, for all $p \in \Gamma$, $t \in \mathbb{R}$ and $x \in \mathcal{X}$, showing that $T$ is the generator of a $\Gamma$-isometric group. In particular, $T$ is the generator of an equicontinuous group.

(2) Suppose $T$ is a $\Gamma$-dissipative operator which is the generator of a $\Gamma$-semigroup. By \cite[Theorem 4.2]{babalola}, there exists a projective family $\left\{T_p\right\}_{p \in \Gamma}$ of closable linear operators such that $\overline{T_p}$ is the generator of a strongly continuous one-parameter semigroup $\tilde{V}_p \colon t \longmapsto \tilde{V}_p(t)$ on $\mathcal{X}_p$, for each $p \in \Gamma$, and $T = \varprojlim \overline{T_p}$ is the generator of the projective limit semigroup $t \longmapsto V(t)$ corresponding to $\left\{\tilde{V}_p\right\}_{p \in \Gamma}$. Since $T$ is a $\Gamma$-dissipative operator on $\mathcal{X}$ it follows that, for each $p \in \Gamma$, $\overline{T_p}$ is a dissipative operator on $\mathcal{X}_p$. Fix $p \in \Gamma$. Since $\overline{T_p}$ generates a strongly continuous semigroup it follows from the Feller-Miyadera-Phillips Theorem \cite[Theorem 3.8, page 77]{engel}, in particular, that there exists $\sigma_p \geq 0$ such that, for all $\lambda > \sigma_p$, one has $\text{Ran }(\lambda I - \overline{T_p}) = \mathcal{X}_p$. Therefore, by \cite[Proposition 3.14 (ii), (iv), page 82]{engel}, it follows that $\overline{\text{Ran }(\lambda I - T_p)} = \text{Ran }(\lambda I - \overline{T_p}) = \mathcal{X}_p$, for all $\lambda > 0$. Finally, with the aid of the Lumer-Phillips Theorem \cite[Theorem 3.15, page 83]{engel}, one concludes that $\overline{T_p}$ actually generates a contraction semigroup. Hence, \begin{equation*} p(V(t)x) = \|\tilde{V}_p(t)([x]_p)\|_p \leq \|[x]_p\|_p = p(x), \qquad p \in \Gamma, \, t \geq 0, \, x \in \mathcal{X},\end{equation*} and the conclusion that $T$ generates a $\Gamma$-contractive semigroup follows.

Now, suppose $T$ is a $\Gamma$-conservative operator which is the generator of a $\Gamma$-group. Then, repeating the above arguments for the $\Gamma$-dissipative generators $T$ and $-T$ yields $p(V(t)x) \leq p(x)$, for all $p \in \Gamma$, $t \in \mathbb{R}$ and $x \in \mathcal{X}$. Consequently, $p(V(t)x) = p(x)$, for all $p \in \Gamma$, $t \in \mathbb{R}$ and $x \in \mathcal{X}$, showing that $T$ is, indeed, the generator of a $\Gamma$-isometric group. Also, by Corollary \ref{Corollary4}, $\overline{T^2}$ is the generator of a $\Gamma$-semigroup so, as a result of what was just proved, in order to conclude that $\overline{T^2}$ is the generator of a $\Gamma$-contractive semigroup it is sufficient to show that it is $\Gamma$-dissipative. Fix $p \in \Gamma$ and $x_p \in \text{Dom }T_p^2$. Then, since $T = \varprojlim \overline{T_p}$ is $\Gamma$-conservative, \begin{equation*} \left\|\left(I - \lambda^2 T_p^2\right)x_p\right\|_p = \left\|\left(I - \lambda \, T_p\right)\left(I + \lambda \, T_p\right)x_p\right\|_p \geq \|x_p\|_p\end{equation*} for all $\lambda > 0$, showing that $T_p^2$ is a dissipative operator on $\mathcal{X}_p$, for all $p \in \Gamma$. Consequently, each $\overline{T_p^2}$ is also dissipative. But by the proof of Corollary \ref{Corollary4}, $\overline{T^2} = \varprojlim \overline{T_p^2}$ so $\overline{T^2}$ is, in fact, a $\Gamma$-dissipative operator, as claimed. Therefore, $\overline{T^2}$ is the generator of a $\Gamma$-contractive semigroup.
\end{proof}

\begin{remark}
If $T$ is a $\Gamma$-conservative operator for which $\text{Ran }(\lambda I - T)$ is dense in $\mathcal{X}$, for some nonzero $\lambda \in \mathbb{R}$, then the proof of \cite[Theorem 3.14]{albanese} adapted for groups shows that $T$ is closable and that $\overline{T}$ is the generator of a $\Gamma$-isometric group.
\end{remark}

\section{Examples and Applications}

Let $(\mathcal{X}, \tau)$ be a Banach space, $G$ be a real finite-dimensional Lie group of dimension $d$ with Lie algebra $\mathfrak{g}$ and $V \colon G \longrightarrow \mathcal{L}(\mathcal{X})$ be a strongly continuous representation of $G$ on $\mathcal{X}$. An adaptation of \cite[Proposition 10.1.6, page 263]{schmudgen} to the Banach space context shows that the map $\partial V \colon X \longmapsto \partial V(X) := dV(X)|_{C^\infty(V)}$ is a Lie algebra representation which extends to a representation of (the complexification of) the universal enveloping algebra of $\mathfrak{g}$. Fix a basis $\mathcal{B} := \left\{X_k\right\}_{1 \leq k \leq d}$ for $\mathfrak{g}$ and let $\|\, \cdot \,\|$ be a Banach space norm for $(\mathcal{X}, \tau)$. Since each $dV(X_k)$ is the infinitesimal generator of a strongly continuous one-parameter group on $\mathcal{X}$, it follows from the Feller-Miyadera-Phillips Theorem for groups \cite[page 79]{engel} that there exist $M_k \geq 1$ and $\beta_k \in \mathbb{R}$ such that \begin{equation*}\|[(\lambda - \beta_k)\text{R}(\lambda, \pm dV(X_k))]^n x\| \leq M_k \, \|x\|, \qquad \lambda > \beta_k, \, n \in \mathbb{N}, \, x \in \mathcal{X}.\end{equation*} Proceeding just as in the proof of implication (b) $\Rightarrow$ (a) of \cite[Theorem 3.8, page 77]{engel} gives the conclusion that the maps \begin{equation*} x \longmapsto \|x\|_{k, \pm} := \sup_{\mu > \beta_k} \sup_{n \in \mathbb{N}} \|[(\mu - \beta_k)\text{R}(\mu, \pm dV(X_k))]^n x\|, \qquad 1 \leq k \leq d\end{equation*} on $\mathcal{X}$ are well-defined norms which satisfy \begin{equation*}\|x\| \leq \|x\|_{k, \pm} \leq M_k \, \|x\| \quad \text{and} \quad \|(\lambda - \beta_k)\text{R}(\lambda, \pm dV(X_k))x\|_{k, \pm} \leq \|x\|_{k, \pm},\end{equation*} for all $1 \leq k \leq d$ and $\lambda > \beta_k$. Therefore, in particular, we have obtained $2d$ norms on $\mathcal{X}$, all equivalent to $\|\, \cdot \,\|$ such that, for every fixed $1 \leq k \leq d$, there exists $\beta_k \in \mathbb{R}$ with the property that the inequality \begin{equation*}\|(\lambda I - (\pm \partial V(X_k)))^m x\|_{k, \pm} \geq (|\text{Re }\lambda| - \beta_k)^m \, \|x\|_{k, \pm}\end{equation*} holds, whenever $|\text{Re }\lambda| > \beta_k$, $m \in \mathbb{N}$ and $x \in C^\infty(V)$.

Now, fix $1 \leq k_0 \leq d$. In what follows, the norm $\|\, \cdot \,\|_{k_0, +}$ constructed from the operator $dV(X_{k_0})$ will be denoted simply by $\|\, \cdot \,\|$, in order to simplify the notations (the calculations for the norm $\|\, \cdot \,\|_{k_0, -}$, constructed from $- dV(X_{k_0})$, are analogous).

Equip $C^\infty(V)$ with the topology $\tau_\infty$ defined by the family \begin{equation*}\Gamma_\infty^{k_0} := \left\{\|\, \cdot \,\|_n: n \in \mathbb{N}\right\}\end{equation*} of norms, where \begin{equation*}\|x\|_0 := \|x\|, \qquad dV(X_0) := I\end{equation*} and \begin{equation*}\|x\|_n := \max \left\{\|dV(X_{i_1}) \cdots dV(X_{i_n})x\|: 0 \leq i_j \leq d\right\}, \qquad n \geq 1\end{equation*} -- this is called the \textit{projective C$^\infty$-topology} on $C^\infty(V)$, and it does not depend upon the fixed basis $\mathcal{B}$. Also, if $\|\, \cdot \,\|'$ and $\|\, \cdot \,\|''$ are two equivalent norms on $\mathcal{X}$ generating the topology $\tau$ and satisfying $M_1 \, \|\, \cdot \,\|'' \leq \|\, \cdot \,\|' \leq M_2 \, \|\, \cdot \,\|''$, for certain $M_1, M_2 > 0$, then $M_1 \, \|\, \cdot \,\|_n'' \leq \|\, \cdot \,\|_n' \leq M_2 \, \|\, \cdot \,\|_n''$, for every $n \in \mathbb{N}$. Therefore, $\tau_\infty$ is independent of the choice of the Banach space norm, as long as an equivalent norm is chosen.

As a consequence of the Uniform Boundedness Principle, every strongly continuous one-parameter group on a Banach space is locally equicontinuous \cite[Proposition 5.5, page 39]{engel}, so each generator $dV(X)$, $X \in \mathfrak{g}$, is a closed densely defined linear operator on $\mathcal{X}$, as mentioned in the paragraph right after Definition \ref{smooth}. A straightforward adaptation of \cite[Theorem 1.1]{goodman} shows that \begin{equation*}C^\infty(V) = \bigcap_{k = 1}^d \bigcap_{n = 1}^{+ \infty} \text{Dom }dV(X_k)^n.\end{equation*} As a consequence of this description of $C^\infty(V)$, it is clear that $(C^\infty(V), \tau_\infty)$ is a Fr\'echet space -- to prove this, adapt the argument of \cite[Corollary 1.1]{goodman}, exploring the closedness of the operators $dV(X_k)$, $1 \leq k \leq d$. Moreover, $C^\infty(V)$ is a dense subspace of $\mathcal{X}$, since it contains the so called G\aa rding domain \begin{equation*} D_G(V) := \text{span}_{\mathbb{C}} \left\{\int_G \phi(g) \, V(g) x \, dg, \, \phi \in C_c^\infty(G), x \in \mathcal{X}\right\},\end{equation*} which is dense in $\mathcal{X}$ -- see \cite{garding}. Also, the operators \begin{equation*}V_\infty(g) := V(g)|_{C^\infty(V)} \colon C^\infty(V) \longrightarrow C^\infty(V), \qquad g \in G,\end{equation*} are continuous with respect to $\tau_\infty$ \cite[page 90]{poulsen}.

Define the family $\Gamma_\infty$ of norms on $C^\infty(V)$ just as above, but with $\|\, \cdot \,\|_0 = \|\, \cdot \,\|_{k_0, +}$ replaced by the original norm of $\mathcal{X}$. As an application of the theorems in Section 3, we are going to show that $V_\infty \colon g \longmapsto V_\infty(g)$ is a strongly continuous representation of $G$ on $(C^\infty(V), \tau_\infty)$, with each $\partial V(X)$ being the generator of the $\Gamma_\infty$-group $t \longmapsto V_\infty(\exp tX)$, and provide sufficient conditions for these generators to be of bounded type.

Theorem 2 of \cite{harish-chandra} shows that, given $X \in \mathfrak{g}$ and $x \in C^\omega(V)$, there exists $r_x > 0$ such that the series \begin{equation*}\sum_{m = 0}^{+ \infty} \frac{\|\partial V(X)^m x\|}{m!} \, |u|^m,\end{equation*} converges, if $|u| < r_x$. Since $C^\omega(V)$ is left invariant by the operators $\partial V(X)$, $X \in \mathfrak{g}$ (see \cite[page 209]{harish-chandra}), it is possible to iterate the calculations of Observations 1 and 2 inside the proof of \cite[Theorem 3.1]{jorgensengoodman} to obtain, respectively:

(a) for each fixed $X \in \mathfrak{g}$, $n \in \mathbb{N}$ and $x \in C^\omega(V)$, the series \begin{equation*}\sum_{m = 0}^{+ \infty} \frac{\|\partial V(X)^m x\|_n}{m!} \, |u|^m,\end{equation*} converges for sufficiently small $|u|$ (with a radius of convergence depending on $x$ and $n$);

(b) for every $n \in \mathbb{N}$, $x \in C^\infty(V)$ and $\lambda \in \mathbb{C}$ satisfying $|\text{Re }\lambda| > l_n(k_0) := \beta_{k_0} + n \, \mu_{k_0}$, one has \begin{equation} \label{conservativetype} \|(\lambda I - \partial V(X_{k_0}))^m x\|_n \geq (|\text{Re }\lambda| - l_n(k_0))^m \, \|x\|_n, \qquad m \in \mathbb{N},\end{equation} where $\mu_{k_0} \equiv \mu(X_{k_0})$ is the operator norm of $\text{ad }\partial V(X_{k_0}) := [\partial V(X_{k_0}), \, \cdot \,\,]$, when seen as a linear operator on $(\partial V[\mathfrak{g}], \|\, \cdot \,\|_1)$, with $\left\|\sum_{j = 1}^d c_j \, \partial V(X_j)\right\|_1 := \sum_{j = 1}^d |c_j|$ -- if $\left\{\partial V(X_j)\right\}_{1 \leq j \leq d}$ is not a linearly independent set, then one may extract from it a basis for $\partial V[\mathfrak{g}]$; here, we have assumed, without loss of generality, that $\left\{\partial V(X_j)\right\}_{1 \leq j \leq d}$ is already a basis, so that $\|\, \cdot \,\|_1$ is a well-defined norm on $\partial V[\mathfrak{g}]$.\footnote{Note that the definition of $\|\, \cdot \,\|_1$ employed in \cite{jorgensengoodman} is slightly different from ours, so the constants $l_n(k_0)$ must be adjusted accordingly.}

In particular, (a) shows that \begin{equation*}C^\omega(V) \subseteq \cap_{k = 1}^d C^\omega_\leftarrow(\partial V(X_k)).\end{equation*} Also, note that $C^\omega_\leftarrow(\partial V(X_k)) \subseteq C^\infty(V)$, $1 \leq k \leq d$. By \cite[Theorem 1.3]{poulsen}, $C^\omega(V)$ is $\tau_\infty$-dense in $C^\infty(V)$, since it is a $\tau$-dense subspace of $\mathcal{X}$ \cite[Theorem 4]{nelson} which satisfies $V(g)[C^\omega(V)] \subseteq C^\omega(V)$, for all $g \in G$. Hence, it follows that $C^\omega_\leftarrow(\partial V(X_{k_0}))$ is $\tau_\infty$-dense in $C^\infty(V) = \text{Dom }\partial V(X_{k_0})$. The operator $\partial V(X_{k_0})$ also has the (KIP) with respect to $\Gamma_\infty^{k_0}$ (since $\Gamma_\infty^{k_0}$ consists of norms), a fact which, combined with (b), shows that its closure $\overline{\partial V(X_{k_0})}^\infty$ with respect to $\tau_\infty$ is the generator of a $\Gamma_\infty^{k_0}$-group $V_{k_0} \colon t \longmapsto V_{k_0}(t)$ on $C^\infty(V)$, by Theorem \ref{Theorem1}. Since $\partial V(X)$ is a continuous operator on $(C^\infty(V), \tau_\infty)$, for all $X \in \mathfrak{g}$, the equality $\partial V(X_{k_0}) = \overline{\partial V(X_{k_0})}^\infty$ follows. But there is a subtle point which must be mentioned: the fundamental system of norms $\Gamma_\infty^{k_0}$ was constructed with respect to the norm $\|\, \cdot \,\|_0 = \|\, \cdot \,\|_{k_0, +}$, so while $\|V_{k_0}(t)x\|_n \leq e^{l_n(k_0) |t|} \|x\|_n$, for all $n \in \mathbb{N}$, $t \in \mathbb{R}$ and $x \in C^\infty(V)$, for this choice of the Banach space norm, one has for the \textit{original} Banach space norm $\|\, \cdot \,\|_0 = \|\, \cdot \,\|$ the estimates $\|V_{k_0}(t)x\|_n \leq M_{k_0} \, e^{l_n(k_0) |t|} \|x\|_n$, where $M_{k_0} > 0$, $n \in \mathbb{N}$, $t \in \mathbb{R}$ and $x \in C^\infty(V)$. This is a consequence of the already mentioned fact that two equivalent norms $\|\, \cdot \,\|'$ and $\|\, \cdot \,\|''$ on $\mathcal{X}$ satisfying $M_1 \, \|\, \cdot \,\|'' \leq \|\, \cdot \,\|' \leq M_2 \, \|\, \cdot \,\|''$, for certain $M_1, M_2 > 0$, also satisfy $M_1 \, \|\, \cdot \,\|_n'' \leq \|\, \cdot \,\|_n' \leq M_2 \, \|\, \cdot \,\|_n''$, for every $n \in \mathbb{N}$; the constant $M_{k_0}$, above, appears as a consequence of this reasoning -- use the estimates (\ref{conservativetype}) combined with Theorem \ref{Theorem1}. Therefore, since $1 \leq k_0 \leq d$ is arbitrary, invoking Theorem \ref{Theorem1} again we conclude that $\partial V(X_k)$ is the generator of a $\Gamma_\infty$-group $V_k \colon t \longmapsto V_k(t)$ on $C^\infty(V)$, for all $1 \leq k \leq d$.

Next, we prove that $V_k(t) = V_\infty(\exp t X_k)$, for every $t \in \mathbb{R}$ and $1 \leq k \leq d$. First, note that a $\Gamma_\infty$-group $\Psi \colon t \longmapsto \Psi(t)$ on $(C^\infty(V), \tau_\infty)$ can always be extended to a strongly continuous one-parameter group $\tilde{\Psi} \colon t \longmapsto \tilde{\Psi}(t)$ on $(\mathcal{X}, \tau)$. In fact, for each fixed $t \in \mathbb{R}$, define $\tilde{\Psi}(t)$ as the unique continuous linear operator extending $\Psi(t)$ defined on all of $\mathcal{X}$. Let $\epsilon > 0$ and fix $x_0 \in \mathcal{X}$. From the definition of $\tilde{\Psi}(t)$ and the fact that $\Psi \colon t \longmapsto \Psi(t)$ is a $\Gamma_\infty$-group on $(C^\infty(V), \tau_\infty)$, it follows that there exist $M_0, \sigma_0 \in \mathbb{R}$ such that \begin{equation*} \|\tilde{\Psi}(t)x\| \leq M_0 \, e^{\sigma_0 |t|} \|x\|, \qquad t \in \mathbb{R}, \, x \in \mathcal{X}.\end{equation*} Define $\tilde{M}_0 := \max \left\{M_0 \, e^{\sigma_0}, 1\right\}$. Since $C^\infty(V)$ is dense in $\mathcal{X}$, there exists $y \in C^\infty(V)$ such that $\|x_0 - y\| < \epsilon / (3 \, \tilde{M}_0)$. Also, because of the strong continuity of $\Psi$ on $(C^\infty(V), \tau_\infty)$ one guarantees, in particular, the existence of $0 < r_0 \leq 1$ with the property that $\|\tilde{\Psi}(t)y - y\| < \epsilon / 3$, for all $|t| < r_0$. Hence, if $|t| < r_0$, \begin{equation*} \|\tilde{\Psi}(t)x_0 - x_0\| \leq \sup_{t \in [-1, 1]} \|\tilde{\Psi}(t)(x_0 - y)\| + \|\tilde{\Psi}(t)y - y\| + \|y - x_0\|\end{equation*} \begin{equation*} < M_0 \, e^{\sigma_0} \|x_0 - y\| + \epsilon / 3 + \epsilon / (3 \, \tilde{M}_0) \leq \epsilon.\end{equation*} This shows $\tilde{\Psi} \colon t \longmapsto \tilde{\Psi}(t)$ is a strongly continuous one-parameter group on $(\mathcal{X}, \tau)$, as claimed.

Therefore, for all $1 \leq k \leq d$, the one-parameter group $V_k$ on $C^\infty(V)$ can be extended to a strongly continuous one-parameter group $\tilde{V}_k \colon t \longmapsto \tilde{V}_k(t)$ on $(\mathcal{X}, \tau)$, which is generated by an operator $T_k$. Fix $1 \leq k \leq d$. Since $\tau_\infty$ is finer than $\tau$, $T_k|_{C^\infty(V)} = \partial V(X_k)$. By \cite[Corollary 1.2]{poulsen} (or even Lemma \ref{Lemma}), it follows that $C^\infty(V) \subseteq C^\infty(\tilde{V}_k)$ is a core for $T_k$. Hence, \begin{equation*} T_k = \overline{T_k|_{C^\infty(V)}} = \overline{\partial V(X_k)} = dV(X_k).\end{equation*} Now, because two strongly continuous one-parameter groups on a Banach space having the same infinitesimal generator must be equal \cite[Theorem 1.4, page 51]{engel} one has, in particular, that $V_k(t) = V_\infty(\exp t X_k)$, for every $t \in \mathbb{R}$.

There exist $d$ real-valued analytic functions $\left\{t_k\right\}_{1 \leq k \leq d}$ defined on a relatively compact open neighborhood $\Omega$ of the identity of $G$ such that $g \longmapsto (t_k(g))_{1 \leq k \leq d}$ maps $\Omega$ diffeomorphically onto a neighborhood of the origin of $\mathbb{R}^d$, with \begin{equation*}g = \exp (t_1(g) \, X_1) \ldots \exp (t_k(g) \, X_k) \ldots \exp (t_d(g) \, X_d), \qquad g \in \Omega,\end{equation*} so \begin{equation*}V_\infty(g) = V_\infty(\exp (t_1(g) \, X_1)) \ldots V_\infty(\exp (t_k(g) \, X_k)) \ldots V_\infty(\exp (t_d(g) \, X_d))\end{equation*} on this neighborhood. This decomposition, combined with the local equicontinuity of each $\Gamma_\infty$-group $t \longmapsto V_\infty(\exp t X_k)$, establishes the strong continuity of $V_\infty$ with respect to $\tau_\infty$.

Now, choosing an adequate norm $\|\, \cdot \,\|_X$ for each fixed $X \in \mathfrak{g}$, just like it was done with the basis elements, and repeating the above reasoning, one sees that \begin{equation*}C^\omega(V) \subseteq  C^\omega_\leftarrow(\partial V(X)) \subseteq C^\infty(V),\end{equation*} with each $\partial V(X)$ being the generator of the $\Gamma_\infty$-group $t \longmapsto V_\infty(\exp tX)$ -- if $\partial V(X)$ already has the property that, for some $\beta(X) \in \mathbb{R}$ and every $\lambda \in \mathbb{C}$ satisfying $|\text{Re }\lambda| > \beta(X)$, \begin{equation*}\|(\lambda I - \partial V(X))^m x\| \geq (|\text{Re }\lambda| - \beta(X))^m \, \|x\|, \qquad x \in C^\infty(V), \, m \in \mathbb{N},\end{equation*} then $\|\, \cdot \,\|$ is already an ``adequate'' norm, so one may take $\|\, \cdot \,\|_X := \|\, \cdot \,\|$. Consequently, for every $n \in \mathbb{N}$, $x \in C^\infty(V)$ and $\lambda \in \mathbb{C}$ satisfying $|\text{Re }\lambda| > l_n(X) := \beta(X) + n \, \mu(X)$, one has \begin{equation*}\|(\lambda I - \partial V(X))^m x\|_n \geq \frac{(|\text{Re }\lambda| - l_n(X))^m}{M_X} \, \|x\|_n, \qquad m \in \mathbb{N},\end{equation*} for certain numbers $\beta(X) \in \mathbb{R}$ and $M_X > 0$, where $\mu(X)$ is the operator norm of $\text{ad }\partial V(X) := [\partial V(X), \, \cdot \,\,]$ when seen as a linear operator on $(\partial V[\mathfrak{g}], \|\, \cdot \,\|_1)$. Note also that, by Corollary \ref{Corollary4}, $[\partial V(X)]^2 = \overline{[\partial V(X)]^2}^\infty$ is the generator of a $\Gamma_\infty$-semigroup.

Since \begin{equation*}\sup_{n \in \mathbb{N}} l_n(X) < + \infty \quad \text{ if, and only if, } \quad \mu(X) = 0, \end{equation*} it follows that if the type of $t \longmapsto V(\exp tX)$ is not $- \infty$ (so that it is a real number) and $\mu(X) = 0$, then $t \longmapsto V_\infty(\exp tX)$ is of bounded type. If $\mu(X) = 0$ and $t \longmapsto V(\exp tX)$ is an isometric representation on $(\mathcal{X}, \|\, \cdot \,\|)$, meaning $\|V(\exp tX)x\| = \|x\|$, for all $t \in \mathbb{R}$ and $x \in \mathcal{X}$, then as a consequence of the above calculations combined with Theorem \ref{Corollary6}(2), $t \longmapsto V_\infty(\exp tX)$ is a $\Gamma_\infty$-isometric group on $C^\infty(V)$, since $\partial V(X)$ is $\Gamma_\infty$-conservative. In this case, $[\partial V(X)]^2$ is the generator of a $\Gamma_\infty$-contractive semigroup, also by Theorem \ref{Corollary6}(2). Summarizing:

\begin{theorem} \label{examples}
If $(\mathcal{X}, \|\, \cdot \,\|)$ is a Banach space and $V \colon G \longrightarrow \mathcal{L}(\mathcal{X})$ is a strongly continuous representation on $\mathcal{X}$, then it restricts to a strongly continuous representation $V_\infty \colon G \longrightarrow \mathcal{L}(C^\infty(V))$ on $(C^\infty(V), \tau_\infty)$ which is implemented by (one-parameter) $\Gamma_\infty$-groups, where $\tau_\infty$ is the topology generated by the family $\Gamma_\infty := \left\{\|\, \cdot \,\|_n: n \in \mathbb{N}\right\}$ of norms on $C^\infty(V)$ defined by \begin{equation*}\|x\|_0 := \|x\|, \qquad dV(X_0) := I\end{equation*} and \begin{equation*}\|x\|_n := \max \left\{\|dV(X_{i_1}) \cdots dV(X_{i_n})x\|: 0 \leq i_j \leq d\right\}, \qquad n \geq 1.\end{equation*} The space $C^\omega(V)$ of analytic vectors for $V$ is a $\tau_\infty$-dense subspace of $C^\infty(V)$ consisting of $\tau_\infty$-projective analytic vectors for the generator $\partial V(X)$ of the $\Gamma_\infty$-group $t \longmapsto V_\infty(\exp tX)$, for all $X \in \mathfrak{g}$. Also, $[\partial V(X)]^2$ is the generator of a $\Gamma_\infty$-semigroup, for all $X \in \mathfrak{g}$.

Moreover, if $X \in \mathfrak{g}$ satisfies $\mu(X) = 0$ and the type of $t \longmapsto V(\exp tX)$ is not $- \infty$, then the $\Gamma_\infty$-group $t \longmapsto V_\infty(\exp tX)$ is of bounded type. In other words, if the type of $t \longmapsto V(\exp tX)$ is not $- \infty$ and $\partial V(X)$ belongs to the center of the Lie algebra $\partial V[\mathfrak{g}]$, then the $\Gamma_\infty$-group $t \longmapsto V_\infty(\exp tX)$ is of bounded type.

If $X \in \mathfrak{g}$ satisfies $\mu(X) = 0$ and $t \longmapsto V(\exp tX)$ is an isometric representation on $(\mathcal{X}, \|\, \cdot \,\|)$, then $t \longmapsto V_\infty(\exp tX)$ is a $\Gamma_\infty$-isometric group on $C^\infty(V)$ and $[\partial V(X)]^2$ is the generator of a $\Gamma_\infty$-contractive semigroup on $C^\infty(V)$.
\end{theorem}

For an example of this situation, consider the Schwartz function space $\mathcal{S}(\mathbb{R}^n)$ as a subspace of $L^2(\mathbb{R}^n)$ and equip it with the family $\Gamma_\infty$ of norms given by $\|\, \cdot \,\|_m \colon f \longmapsto \max_{|\alpha| \leq m, |\beta| \leq m} \|x^\alpha \partial^\beta f\|_2$, where $\|\, \cdot \,\|_2$ denotes the L$^2$-norm, $m \in \mathbb{N}$ and $\alpha, \beta \in \mathbb{N}^n$ are multiindices. Also, define a matrix group by \begin{equation*}H_{2n + 1}(\mathbb{R}) = \left\{\begin{bmatrix}
1 & \textbf{a}^T & c\\
0 & I_n & \textbf{b}\\
0 & 0 & 1
\end{bmatrix}: \textbf{a}, \textbf{b} \in \mathbb{R}^n, c \in \mathbb{R}\right\}\end{equation*} with the usual matrix multiplication, where $I_n$ is the identity matrix of $M_n(\mathbb{R})$. This is known as the Heisenberg group of dimension $2n + 1$. Define a strongly continuous unitary representation $U$ of $H_{2n + 1}(\mathbb{R})$ on $L^2(\mathbb{R}^n)$ by \begin{equation*}U_{\textbf{a}, \textbf{b}, c} := U\begin{bmatrix}
1 & \textbf{a}^T & c\\
0 & I_n & \textbf{b}\\
0 & 0 & 1
\end{bmatrix} \colon f \longmapsto e^{ic} e^{i \langle \textbf{b}, \cdot \, \rangle} f(\, \cdot \, - \textbf{a}).\end{equation*} Then, its Lie algebra is sent onto the ($2n + 1$)-dimensional real Lie algebra \begin{equation*}\mathcal{L} := \text{span}_\mathbb{R} \left\{\partial_k \Big|_{\mathcal{S}(\mathbb{R}^n)}, i x_k|_{\mathcal{S}(\mathbb{R}^n)}, i|_{\mathcal{S}(\mathbb{R}^n)}: 1 \leq k \leq n\right\}\end{equation*} via the Lie algebra representation $\partial U$, since $C^\infty(U) = \mathcal{S}(\mathbb{R}^n)$, and it is a realization of the Canonical Commutation Relations (CCR) by unbounded operators on $L^2(\mathbb{R}^n)$. An easy adaptation of the calculations above shows, in particular, that the linear operators $c_0 \, i + \sum_{k = 1}^n c_k \, \partial_k + \sum_{k = 1}^n d_k \, i x_k$, $c_k, d_k \in \mathbb{R}$, all generate $\Gamma_\infty$-groups on $\mathcal{S}(\mathbb{R}^n)$, thus complementing some of the examples given in \cite{babalola}. Moreover, their squares generate $\Gamma_\infty$-semigroups on $\mathcal{S}(\mathbb{R}^n)$.

For another example, consider the torus $\mathbb{T}^n := \mathbb{R}^n/(2\pi \mathbb{Z})^n$ and, for each $y \in \mathbb{T}^n$, let $T_y$ denote the unitary operator on $L^2(\mathbb{T}^n)$ defined by $(T_y \, u)(x) = u(x - y)$. For each $j \in \mathbb{Z}^n$, let $e_j \in C^\infty(\mathbb{T}^n)$ be defined by $e_j(x) = e^{i \langle j, x \rangle}$, where we denote by the same letter $x$ both an element $x$ of $\mathbb{R}^n$ and its class $[x] \in \mathbb{T}^n$. Then, by \cite[Theorem 2]{cabralmelo}, a bounded operator $A \in \mathcal{L}(L^2(\mathbb{T}^n))$ is such that the map $\mathbb{T}^n \ni y \longmapsto T_y A T_{-y}$ is smooth with respect to the norm topology of $\mathcal{L}(L^2(\mathbb{T}^n))$ if, and only if, $A = \text{Op}(a_j)$ for some symbol $(a_j)_{j \in \mathbb{Z}^n}$ of order zero: in other words if, and only if, $A$ is a bounded operator on $L^2(\mathbb{T}^n)$ defined by \begin{equation*}Au(x) = \frac{1}{(2 \pi)^n} \sum_{j \in \mathbb{Z}^n} a_j(x) e_j(x) \widehat{u}_j, \qquad \widehat{u}_j := \int_{\mathbb{T}^n} e_{-j} u,\end{equation*} for all $u \in C^\infty(\mathbb{T}^n)$ and $x \in \mathbb{T}^n$, with $(a_j)_{j \in \mathbb{Z}^n}$ satisfying $a_j \in C^\infty(\mathbb{T}^n)$ and \begin{equation*}\sup \left\{|\partial^\alpha a_j(x)|; j \in \mathbb{Z}^n, x \in \mathbb{T}^n\right\}<\infty,\end{equation*} for every multiindex $\alpha \in \mathbb{N}^n$. For such $A = \text{Op}(a_j)$, one has \begin{equation*}T_y A T_{-y} = \text{Op}((T_y \, a_j)_{j \in \mathbb{Z}^n}),\end{equation*} so $\partial_y^\alpha (T_y A T_{-y}) = T_y [\text{Op}((\partial^\alpha a_j)_{j \in \mathbb{Z}^n})] T_{-y}$, for every $\alpha \in \mathbb{N}^n$ and $y \in \mathbb{T}^n$. This last equality can be proved by the repeated use of the equality \begin{equation*} a_j(x + h f_k) - a_j(x) - h \, \partial_k a_j(x) = h^2 \int_0^1 t \int_0^1 \partial_k^2 a_j(x + t s h f_k) ds \, dt,\end{equation*} where $j \in \mathbb{Z}^n$, $h \in \mathbb{R}$ is small enough and $f_k$ denotes the $k^{\text{th}}$ element of the canonical basis of $\mathbb{R}^n$, combined with estimates \cite[Theorem 1 -- (4)]{cabralmelo}. Therefore, the infinitesimal generators of the (not everywhere strongly continuous) adjoint representation $V \colon y \longmapsto T_y \, (\, \cdot \,) \, T_{-y}$ restrict to the $*$-algebra of smooth operators $C^\infty(V)$ as operators of the form \begin{equation*}\text{Op}(a_j) \longmapsto \text{Op}\left(\sum_{1 \leq k \leq n} c_k \, \partial_k a_j\right),\end{equation*} $c_k \in \mathbb{R}$, and are all generators of $\Gamma_\infty$-groups, where $\Gamma_\infty := \left\{\|\, \cdot \,\|_n: n \in \mathbb{N}\right\}$ with $\|\, \cdot \,\|_0 = \|\, \cdot \,\|$ being the usual operator norm. Actually, they are generators of $\Gamma_\infty$-groups of bounded type, since $V$ has type 0 and the range of $\partial V$ is an abelian Lie algebra: more precisely, by Theorem \ref{examples} they are generators of $\Gamma_\infty$-isometric groups, since $y \longmapsto T_y \, (\, \cdot \,) \, T_{-y}$ is an isometric representation on $\mathcal{L}(L^2(\mathbb{T}^n))$. Moreover, their squares are generators of $\Gamma_\infty$-contractive semigroups. As a consequence of \cite[Theorem 3.3]{babalola} and the calculations above, the resolvent operator of $\partial V(X)$ exists for $\lambda \in \mathbb{C} \backslash i \, \mathbb{R}$ and belongs to $\mathcal{L}_{\Gamma_\infty}(C^\infty(V))$, for every fixed $X$ in the Lie algebra of $\mathbb{T}^n$, with \begin{equation} \label{resolventtorus} \|\text{R}(\lambda, \partial V(X))A\|_n \leq \frac{1}{|\text{Re }\lambda|} \, \|A\|_n, \qquad \text{Re }\lambda \neq 0, \, n \in \mathbb{N}, \, A \in C^\infty(V).\end{equation} Another interesting fact is that the analytic vectors with respect to the representation $y \longmapsto T_y \, (\, \cdot \,) \, T_{-y}$, characterized in \cite[Theorem 3]{cabralmelo} as a certain algebra of pseudodifferential operators acting on $L^2(\mathbb{T}^n)$, form a $\tau_\infty$-dense $*$-subalgebra of $C^\infty(V)$ consisting of $\tau_\infty$-projective analytic vectors for the operators $\text{Op}(a_j) \longmapsto \text{Op}\left(\sum_{1 \leq k \leq n} c_k \, \partial_k a_j\right)$, $c_k \in \mathbb{R}$, where $(a_j)_{j \in \mathbb{Z}^n}$ satisfies $a_j \in C^\infty(\mathbb{T}^n)$ and \begin{equation*}\sup \left\{|\partial^\alpha a_j(x)|; j \in \mathbb{Z}^n, x \in \mathbb{T}^n\right\}<\infty,\end{equation*} for every multiindex $\alpha \in \mathbb{N}^n$.\footnote{A characterization of the smooth vectors for the circle was first obtained in \cite{toscano}, which is a discrete version of a characterization obtained in \cite{cordes2} -- see also \cite[Chapter 8]{cordes}.}

An analogous application can be given for the canonical (not everywhere strongly continuous) action $\theta \colon \begin{bmatrix}
1 & \textbf{a}^T & c\\
0 & I_n & \textbf{b}\\
0 & 0 & 1
\end{bmatrix} \longmapsto U_{\textbf{a}, \textbf{b}, c} \, (\, \cdot \,) \, (U_{\textbf{a}, \textbf{b}, c})^{-1}$ of the Heisenberg group $H_{2n + 1}(\mathbb{R})$ on the C$^*$-algebra of bounded operators $\mathcal{L}(L^2(\mathbb{R}^n))$: H.O. Cordes proved in \cite{cordes2} that a bounded linear operator $A$ on $L^2(\mathbb{R}^n)$ is such that \begin{equation*} H_{2n + 1}(\mathbb{R}) \ni \begin{bmatrix}
1 & \textbf{a}^T & c\\
0 & I_n & \textbf{b}\\
0 & 0 & 1
\end{bmatrix} \longmapsto U_{\textbf{a}, \textbf{b}, c} \, A \, (U_{\textbf{a}, \textbf{b}, c})^{-1} \in \mathcal{L}(L^2(\mathbb{R}^n))\end{equation*} is a smooth function with values in the C$^*$-algebra of all bounded operators on $L^2(\mathbb{R}^n)$ if, and only if, there exists $a \in C^\infty(\mathbb{R}^{2n})$, bounded and with all its partial derivatives also bounded -- denote this space by $\mathcal{C}\mathcal{B}^\infty(\mathbb{R}^{2n})$ -- such that, for all $u \in \mathcal{S}(\mathbb{R}^n)$ and all $x \in \mathbb{R}^n$, one has \begin{equation*}Au(x) = \frac{1}{(2 \pi)^n} \int_{\mathbb{R}^n} e^{i x \cdot \xi} \, a(x, \xi) \, \widehat{u}(\xi) \, d\xi, \quad \text{with} \quad \widehat{u}(\xi) := \int_{\mathbb{R}^n} e^{-i s \cdot \xi} \, u(s) \, ds.\end{equation*} In this case, such $A$ is denoted by $a(x, D)$. Therefore, the operators sending these $a(x, D)$ to $\left(\sum_{k = 1}^{2n} c_k \, \partial_k a\right)(x, D)$, $c_k \in \mathbb{R}$, are generators of $\Gamma_\infty$-groups of bounded type on $C^\infty(\theta)$, since $\theta$ has type 0 and the range of $\partial \theta$ is an abelian Lie algebra (to prove that the first $2n$ canonical directions give rise to generators which act on $a(x, D)$ via a partial differentiation $\partial_k$ on the symbol of $a(x, D)$, $1 \leq k \leq 2n$, adapt the strategy employed above, in the case of the torus, but using the estimates of the Calder\'on-Vaillancourt Theorem \cite{calderon}; note that the $(2n + 1)^{\text{th}}$ direction gives the zero operator as a generator). Just as in the previous example, they are generators of $\Gamma_\infty$-isometric groups, since $\theta$ is an isometric representation on $\mathcal{L}(L^2(\mathbb{R}^n))$. Moreover, their resolvent operators exist for $\lambda \in \mathbb{C} \backslash i \, \mathbb{R}$, belong to $\mathcal{L}_{\Gamma_\infty}(C^\infty(\theta))$ and satisfy estimates which are analogous to (\ref{resolventtorus}), thus complementing the statement of \cite[Proposition 4.2, page 262]{cordes}. Note, however, that the constants appearing in our estimates are independent of the dimension of the subjacent Euclidean space $\mathbb{R}^n$.

Following \cite{cordes}, denote $C^\infty(\theta)$ by $\Psi GT$ and equip the set \begin{equation*}gl := \left\{(g, \textbf{a}, \textbf{b}, c): g \in GL(\mathbb{R}^n), \textbf{a}, \textbf{b} \in \mathbb{R}^n, c \in \mathbb{R}/(2\pi \mathbb{Z})\right\}\end{equation*} with a Lie group structure, as in \cite[page 265]{cordes}. If $\mathcal{T}$ is the unitary representation of $gl$ on $L^2(\mathbb{R}^n)$ defined by \begin{equation*}(\mathcal{T}_{g, \textbf{a}, \textbf{b}, c} \, u)(x) := |\text{det }g|^{1/2} e^{ic} e^{i \langle \textbf{b}, x \rangle} u(gx + \textbf{a}),\end{equation*} then the $*$-algebra of smooth vectors for the adjoint representation \begin{equation*} (g, \textbf{a}, \textbf{b}, c) \longmapsto \mathcal{T}_{g, \textbf{a}, \textbf{b}, c} \,(\, \cdot \,)\, (\mathcal{T}_{g, \textbf{a}, \textbf{b}, c})^{-1}, \end{equation*} denoted by $\Psi GL$, consists precisely of the elements $a(x, D)$ in $\Psi GT$ such that their symbols $a \in \mathcal{C}\mathcal{B}^\infty(\mathbb{R}^{2n})$ remain in $\mathcal{C}\mathcal{B}^\infty(\mathbb{R}^{2n})$ after any finite number of applications of the operators \begin{equation*}\epsilon_{jl} := \xi_j \partial_{\xi_l} - x_l \partial_{x_j}, \qquad j, l = 1, \ldots n, \end{equation*} by \cite[Theorem 5.3, page 269]{cordes}. Moreover, the operators \begin{equation*}a(x, D) \longmapsto \left(\sum_{1 \leq k \leq 2n} c_k \, \partial_k a + \sum_{1 \leq j,l \leq n} d_{jl} \, \epsilon_{jl} \, a\right)(x, D), \qquad c_k, d_{jl} \in \mathbb{R},\end{equation*} on $\Psi GL$ are generators of $\Gamma_\infty$-groups, and their squares generate $\Gamma_\infty$-semigroups. A similar result may be obtained if one substitutes $gl$ by the subgroup \begin{equation*}gs := \left\{(\sigma Q, \textbf{a}, \textbf{b}, c): \sigma > 0, Q \in SO(n), \textbf{a}, \textbf{b} \in \mathbb{R}^n, c \in \mathbb{R}/(2\pi \mathbb{Z})\right\}\end{equation*} and considers the subsequent space of smooth vectors $\Psi GS$ -- see \cite[pages 264--266]{cordes} and \cite[Theorem 5.4, page 269]{cordes}. Note that $\Psi GL \subseteq \Psi GS \subseteq \Psi GT$.\footnote{The author would like to thank Professor Severino T. Melo for suggesting the study of the examples $\Psi GL$ and $\Psi GS$.}

\bibliographystyle{amsplain}

\end{document}